\documentclass{proc-l}
\newtheorem{theorem}{Theorem}[section]
\newtheorem{lemma}[theorem]{Lemma}
\newtheorem{proposition}[theorem]{Proposition}

\theoremstyle{definition}

\theoremstyle{remark}

\numberwithin{equation}{section}
\mathchardef\hyphen="2D
\begin{document}
\title{ON MEAN ERGODIC CONVERGENCE IN THE CALKIN ALGEBRAS}
\author{MARCH T.~BOEDIHARDJO}
\address{DEPARTMENT OF MATHEMATICS, TEXAS A\&M UNIVERSITY, COLLEGE STATION, TEXAS 77843}
\curraddr{}
\email{march@math.tamu.edu}
\thanks{The first author was supported in part by the N.~W.~Naugle Fellowship and the A.~G. \& M.~E.~Owen Chair in the Department of Mathematics.}
\author{WILLIAM B.~JOHNSON}
\address{DEPARTMENT OF MATHEMATICS, TEXAS A\&M UNIVERSITY, COLLEGE STATION, TEXAS 77843}
\curraddr{}
\email{johnson@math.tamu.edu}
\thanks{The second author was supported in part by NSF DMS-1301604.}
\date{}
\commby{Thomas Schlumprecht}
\begin{abstract}
In this paper, we give a geometric characterization of mean ergodic convergence in the Calkin algebras for Banach spaces that have the bounded compact approximation
property.
\end{abstract}

\maketitle
\section{Introduction}
Let $X$ be a real or complex Banach space and let $B(X)$ be the algebra of all bounded linear operators on $X$. Suppose that $T\in B(X)$ and consider the sequence
\[M_{n}(T):=\frac{I+T+\ldots+T^{n}}{n+1},\quad n\geq 1.\]
In \cite{Dunford}, Dunford considered the norm convergence of $(M_{n}(T))_{n}$ and established the following characterizations.
\begin{theorem}\label{thm1.1}
Suppose that $X$ is a complex Banach space and that $T\in B(X)$ satisfies $\frac{\|T^{n}\|}{n}\to 0$. Then the following conditions are equivalent.
\begin{enumerate}
\item $(M_{n}(T))_{n}$ converges in norm to an element in $B(X)$.
\item $1$ is a simple pole of the resolvent of $T$ or $1$ is in the resolvent set of $T$.
\item $(I-T)^{2}$ has closed range.
\end{enumerate}
\end{theorem}
It was then discovered by Lin in \cite{Lin} that $I-T$ having closed range is also an equivalent condition. Moreover, Lin's argument worked also for real Banach spaces.
This result was later improved by Mbekhta and Zem\'anek in \cite{Mbekhta} in which they showed that $(I-T)^{m}$ having closed range, where $m\geq 1$, are also equivalent
conditions. More precisely,
\begin{theorem}\label{thm1.2}
Let $m\geq 1$. Suppose that $X$ is a real or complex Banach space and that $T\in B(X)$ satisfies $\frac{\|T^{n}\|}{n}\to 0$. Then the sequence $(M_{n}(T))_{n}$ converges
in norm to an element in $B(X)$ if and only if $(I-T)^{m}$ has closed range.
\end{theorem}
Let $K(X)$ be the closed ideal of compact operators in $B(X)$. If $T\in B(X)$ then its image in the Calkin algebra $B(X)/K(X)$ is denoted by $\dot{T}$. By
Dunford's Theorem \ref{thm1.1} or by an analogous version for Banach algebras (without condition (3)), when $X$ is a complex Banach space and $\frac{\|\dot{T}^{n}\|}{n}\to 0$, the convergence of $(M_{n}(\dot{T}))_{n}$ in the Calkin algebra is equivalent to 1 being a simple pole
of the resolvent of $\dot{T}$ or being in the resolvent set of $\dot{T}$. But even if we are given that the limit $\dot{P}\in B(X)/K(X)$ exists, there is no obvious
geometric interpretation of $\dot{P}$. In the context of Theorems \ref{thm1.1} and \ref{thm1.2}, if the limit of $(M_{n}(T))_{n}$ exists, then it is a projection onto
$\text{ker}(I-T)$. In the context of the Calkin algebra, the limit $\dot{P}$ is still an idempotent in $B(X)/K(X)$; hence by making a compact perturbation, we can assume
that $P$ is an idempotent in $B(X)$ (see Lemma \ref{lem2.6} below).

A natural question to ask is: what is the range of $P$? Although the range of $P$ is not unique (since $P$ is only unique up to a compact perturbation), it can be
thought of as an analog of $\text{ker}(I-T)$ in the Calkin algebra setting. If $T_{0}\in B(X)$ then $\text{ker }T_{0}$ is the maximal subspace of $X$ on which $T_{0}=0$.
This suggests that the analog of $\text{ker }T_{0}$ in the Calkin algebra setting is the maximal subspace of $X$ on which $T_{0}$ is compact. But the maximal subspace does
not exist unless it is the whole space $X$. Thus, we introduce the following concept.

Let $X$ be a Banach space and let $(P)$ be a property that a subspace $M$ of $X$ may or may not have. We say that a subspace $M\subset X$ is an {\it essentially maximal}
subspace of $X$ satisfying $(P)$ if it has $(P)$ and if every subspace $M_{0}\supset M$ having property $(P)$ satisfies $\text{dim }M_{0}/M<\infty$.

Then the analog of $\text{ker }T_{0}$ in the Calkin algebra setting is an essentially maximal subspace of $X$ on which $T_{0}$ is compact. It turns that if such an analog
for $I-T$ exists, then it is already sufficient for the convergence of $(M_{n}(\dot{T}))_{n}$ in the Calkin algebra (at least for a large class of Banach spaces), which is
the main result of this paper.

Before stating this theorem, we recall that a Banach space $Z$ has the {\it bounded compact approximation property} (BCAP) if there is a uniformly bounded net\\
$(S_{\alpha})_{\alpha\in\Lambda}$ in $K(Z)$ converging strongly to the identity operator $I\in B(Z)$. It is always possible to choose $\Lambda$ to be the set of all finite
dimensional subspaces of $Z$ directed by inclusion.  If the net $(S_{\alpha})_{\alpha\in\Lambda}$ can be chosen so that $\displaystyle\sup_{\alpha\in\Lambda}\|S_{\alpha}\|
\leq\lambda$, then we say that $Z$ has the $\lambda$-BCAP. It is known that if a reflexive space has the BCAP, then the space has the $1$-BCAP. For $T\in B(X)$, the
essential norm $\|T\|_{e}$ is the norm of $\dot{T}$ in $B(X)/K(X)$.
\begin{theorem}\label{thm1.3}
Let $m\geq 1$. Suppose that $X$ is a real or complex Banach space having the bounded compact approximation property. If $T\in B(X)$ satisfies $\frac{\|T^{n}\|_{e}}{n}\to
0$, then the following conditions are equivalent.
\begin{enumerate}
\item The sequence $(M_{n}(\dot{T}))_{n}$ converges in norm to an element in $B(X)/K(X)$.
\item There is an essentially maximal subspace of $X$ on which $(I-T)^{m}$ is compact.
\end{enumerate}
\end{theorem}
The idea of the proof is to reduce Theorem \ref{thm1.3} to Theorem \ref{thm1.2} by constructing a Banach space $\widehat{X}$ and an embedding $f:B(X)/K(X)\to B(\widehat{X}
)$ so that if $T\in B(X)$ and if there is an essentially maximal subspace $M$ of $X$ on which $T$ is compact, then $f(\dot{T})$ has closed range, and then applying Theorem
\ref{thm1.2} to $f(\dot{T})$. The BCAP of $X$ is used to show that $f$ is an embedding but is not used in the construction of $\widehat{X}$ and $f$. The construction of
$f$ is based on the Calkin representation \cite[Theorem 5.5]{Calkin}.

The authors thank C. Foias and C. Pearcy for helpful discussions.
\section{The Calkin representation for Banach spaces}
In this section, $X$ is a fixed infinite dimensional Banach space. Let $\Lambda_{0}$ be the set of all finite dimensional subspaces of $X$ directed by inclusion $\subset$.
Then $\{\{\alpha\in\Lambda_{0}:\alpha\supset\alpha_{0}\}:\alpha_{0}\in\Lambda_{0}\}$ is a filter base on $\Lambda_{0}$, so it is contained in an ultrafilter $U$ on
$\Lambda_{0}$.

Let $Y$ be an arbitary infinite dimensional Banach space and let $(Y^{*})^{U}$ be the ultrapower (see e.g., \cite[Chapter 8]{Diestel}) of $Y^{*}$ with respect to $U$.
(The ultrafilter $U$ and the directed set $\Lambda_{0}$ do not depend on $Y$.) If $(y_{\alpha}^{*})_{\alpha\in\Lambda_{0}}$ is a bounded net in $Y^{*}$, then its image in
$(Y^{
*})^{U}$ is denoted by $(y_{\alpha}^{*})_{\alpha,U}$. Consider the (complemented) subspace\[\widehat{Y}:=\left\{(y_{\alpha}^{*})_{\alpha,U}\in(Y^{*})^{U}:w^{*}\hyphen
\lim_{\alpha,U}y_{\alpha}^{*}=0\right\}\]of $(Y^{*})^{U}$. Here $\displaystyle w^{*}\hyphen\lim_{\alpha,U}y_{\alpha}^{*}$ is the $w^{*}$-limit of $(y_{\alpha}^{*})_{
\alpha\in\Lambda_{0}}$ through $U$, which exists by the Banach-Alaoglu Theorem.

Whenever $T\in B(X,Y)$, we can define an operator $\widehat{T}\in B(\widehat{Y},\widehat{X})$ by sending $(y_{\alpha}^{*})_{\alpha,U}$ to $(T^{*}y_{\alpha}^{*})_{\alpha,
U}$. Note that if $K\in K(X,Y)$ then $\widehat{K}=0$, where $K(X,Y)$ denotes the space of all compact operators in $B(X,Y)$.

\begin{theorem}\label{thm2.1}
Suppose that $X$ has the $\lambda$-BCAP. Then the operator\\$f:B(X)/K(X)\to B(\widehat{X}), \;\dot{T}\mapsto\widehat{T}$, is a norm one
$(\lambda+1)$-embedding into $B(\widehat{X})$ satisfying\[f(\dot{I})=I\text{ and }f(\dot{T_{1}}\dot{T_{2}})=f(\dot{T_{2}})f(\dot{T_{1}}),\qquad T_{1},T_{2}\in B(X).\]
\end{theorem}
\begin{proof}
It is easy to verify that $f$ is a linear map, $f(\dot{I})=I$, and $f(\dot{T_{1}}\dot{T_{2}})=f(\dot{T_{2}})f(\dot{T_{1}})$ for $T_{1},\,T_{2}\in B(X)$. If $T\in B(X)$,
then clearly $\|f(\dot{T})\|\leq\|T\|$, and thus we also have $\|f(\dot{T})\|\leq\|T\|_{e}$. Hence $\|f\|\leq 1$. It remains to show that $f$ is a
$(\lambda+1)$-embedding (i.e., $\displaystyle\inf_{\|T\|_{e}>1}\|f(\dot{T})\| \geq (\lambda+1)^{-1}$).

To do this, let $T\in B(X)$ satisfy $\|T\|_{e}> 1$. Since $X$ has the $\lambda$-BCAP, we can find a net of operators $(S_{\alpha})_{\alpha\in\Lambda_{0}}\subset K(X)$
converging strongly to $I$ such that $\displaystyle\sup_{\alpha\in\Lambda_{0}}\|S_{\alpha}\|\leq\lambda$. Then $\|T^{*}(I-S_{\alpha})^{*}\|=\|(I-S_{\alpha})T\|\geq\|T\|_{e
}> 1,\;\alpha\in\Lambda_{0}$. Thus, there exists $(x_{\alpha}^{*})_{\alpha\in\Lambda_{0}}\subset X^{*}$ such that $\|x_{\alpha}^{*}\|=1$ and $\|T^{*}(I-S_{\alpha})^{*}x_{
\alpha}^{*}\| > 1$ for $\alpha\in\Lambda_{0}$.

Note that for every $x\in X$,\[\limsup_{\alpha\in\Lambda_{0}}|\langle(I-S_{\alpha})^{*}x_{\alpha}^{*},x\rangle|=\limsup_{\alpha\in\Lambda_{0}}|\langle x_{\alpha}^{*},(I-
S_{\alpha})x\rangle|\leq\limsup_{\alpha\in\Lambda_{0}}\|(I-S_{\alpha})x\|=0,\]and so the net $((I-S_{\alpha})^{*}x_{\alpha}^{*})_{\alpha\in\Lambda_{0}}$ converges in the
$w^{*}$-topology to 0. By the construction of $U$, this implies that\[w^{*}\hyphen\lim_{\alpha, U}(I-S_{\alpha})^{*}x_{\alpha}^{*}=0.\]Therefore, due to the definition
$f(\dot{T})=\widehat{T}$, we obtain
\begin{eqnarray*}
(1+\lambda)\|f(\dot{T})\|\geq\|f(\dot{T})\|\lim_{\alpha, U}\|(I-S_{\alpha})^{*}x_{\alpha}^{*}\|&=&\|f(\dot{T})\|\|((I-S_{\alpha})^{*}x_{\alpha}^{*})_{\alpha,U}\|\\&\geq
&\|f(\dot{T})((I-S_{\alpha})^{*}x_{\alpha}^{*})_{\alpha,U}\|\\&=&\lim_{\alpha, U}\|T^{*}(I-S_{\alpha})^{*}x_{\alpha}^{*}\|\geq 1.
\end{eqnarray*}
It follows that $\|f(\dot{T})\|\geq(1+\lambda)^{-1}$ whenever $\|T\|_{e}>1$.
\end{proof}

{\bf Remark 1.}  We do not know whether Theorem \ref{thm2.1} is true without the hypothesis that $X$ has the BCAP.

{\bf Remark 2.}  The embedding in Theorem \ref{thm2.1} is an isometry if the approximating net can be chosen so that $\|I -S_\alpha\| =1$ for every $\alpha$.  This is
the case if, for example, the space $X$ has a 1-unconditional basis.  However, we do not know whether the embedding is an isometry if $X=L_p(0,1)$ with $p\not=2$.

If $N$ is a subset of $Y^{*}$, then we can define a subset $N'$ of $\widehat{Y}$ by\[N':=\left\{(y_{\alpha}^{*})_{\alpha,U}\in\widehat{Y}:\lim_{\alpha,U}d(y_{\alpha}^{*
},N)=0\right\},\]
where\[d(y_{\alpha}^{*},N):=\inf_{z^{*}\in N}\|y_{\alpha}^{*}-z^{*}\|.\]
\begin{lemma}\label{lem2.2}
If $N$ is a $w^{*}$-closed subspace of $Y^{*}$, then for every $(y_{\alpha}^{*})_{\alpha,U}\in\widehat{Y}$,\[d((y_{\alpha}^{*})_{\alpha,U},N')\leq
2\lim_{\alpha,U}d(y_{\alpha}^{*},N).\]
\end{lemma}
\begin{proof}
Let $\displaystyle a=\lim_{\alpha,U}d(y_{\alpha}^{*},N)$. Let $\delta>0$. Then\[A:=\left\{\alpha\in\Lambda:d(y_{\alpha}^{*},N)<a+\delta\right\}\in U.\]
Whenever $\alpha\in A$, $\|y_{\alpha}^{*}-z_{\alpha}^{*}\|<a+\delta$ for some $z_{\alpha}^{*}\in N$. If we take $z_{\alpha}^{*}=0$ for $\alpha\notin A$, then, since
$\displaystyle\sup_{\alpha\in\Lambda}\|y_{\alpha}^{*}\|<\infty$,\[\sup_{\alpha\in\Lambda}\|z_{\alpha}^{*}\|=\sup_{\alpha\in A}\|z_{\alpha}^{*}\|\leq(a+\delta)+\sup_{
\alpha\in A}\|y_{\alpha}^{*}\|<\infty.\]As a consequence, $\displaystyle\left(z_{\alpha}^{*}-w^{*}\hyphen\lim_{\beta,U}z_{\beta}^{*}\right)_ {\alpha,U}\in N'$, since $N$
is $w^{*}$-closed. Therefore,
\begin{eqnarray*}
d\left((y_{\alpha}^{*})_{\alpha,U},N'\right)&\leq& d\left((y_{\alpha}^{*})_{\alpha,U},\left(z_{\alpha}^{*}-w^{*}\hyphen\lim_{\beta,U}z_{\beta}^{*}\right)_{\alpha,U}
\right)\\&=&\lim_{\alpha,U}\left\|y_{\alpha}^{*}-z_{\alpha}^{*}+w^{*}\hyphen\lim_{\beta,U}z_{\beta}^{*}\right\|
\\&\leq&\lim_{\alpha,U}\left\|y_{\alpha}^{*}-z_{\alpha}^{*}\right\|+\left\|w^{*}\hyphen\lim_{\beta,U
 }z_{\beta}^{*}\right\|\\&\leq&(a+\delta)+\left\|w^{*}\hyphen\lim_{\beta,U}(z_{\beta}^{*}-y_{\beta}^{*})\right\|\\&\leq&(a+\delta)+\lim_{\beta,U}\|z_
{\beta}^{*}-y_{\beta}^{*}\|\leq2(a+\delta).
\end{eqnarray*}
But $\delta$ can be arbitarily close to 0 so $\displaystyle d\left((y_{\alpha}^{*})_{\alpha,U},N'\right)\leq 2a=2\lim_{\alpha,U}d(y_{\alpha}^{*},N)$.
\end{proof}
\begin{proposition}\label{pro2.3}
If $X$ and $Y$ are infinite dimensional Banach spaces and if $T\in B(X,Y)$ has closed range then $\widehat{T}\in B(\widehat{Y},\widehat{X}
)$ also has closed range.
\end{proposition}
\begin{proof}
The operator $T$ has closed range so $T^{*}$ also has closed range. Let $\displaystyle c=\inf\{\|T^{*}y^{*}\|:y^{*}\in Y^{*},\;d(y^{*},\text{ker }T^{*})=1
\}>0$. Then by Lemma \ref{lem2.2}, for every $(y_{\alpha}^{*})_{\alpha,U}\in\widehat{Y}$,\[\|\widehat{T}(y_{\alpha}^{*})_{\alpha,U}\|=\lim_{\alpha,U}\|T^{*}y_{\alpha}^{
*}\|\geq c\lim_{\alpha,U}d(y_{\alpha}^{*},\text{ker }T^{*})\geq\frac{c}{2}d((y_{\alpha}^{*})_{\alpha,U},(\text{ker }T^{*})').\]But obviously $(\text{ker }T^{*})'
\subset\text{ker }\widehat{T}$, and so
\[\|\widehat{T}(y_{\alpha}^{*})_{\alpha,U}\|\geq\frac{c}{2}d((y_{\alpha}^{*})_{\alpha,U},\text{ker }\widehat{T}),\quad (y_{\alpha}^{*})_{\alpha,U}\in\widehat{Y}.\]
Hence $\widehat{T}$ has closed range.
\end{proof}
\begin{lemma}\label{lem2.4}
Suppose that $X\subset Y$ and that $T\in B(X)$. Let $T_{0}\in B(X,Y)$, $x\mapsto Tx$. Then $\widehat{T}_{0}\widehat{Y}=\widehat{T}\widehat{X}$.
\end{lemma}
\begin{proof}
If $(y_{\alpha}^{*})_{\alpha,U}\in\widehat{Y}$, then for each $\alpha\in\Lambda$, we have $T_{0}^{*}y_{\alpha}^{*}=T^{*}({y_{\alpha}^{*}}_{|X})$, and $({y_{\alpha
}^{*}}_{|X})_{\alpha,U}\in\widehat{X}$.
Thus $\widehat{T}_{0}(y_{\alpha}^{*})_{\alpha,U}=(T_{0}^{*}y_{\alpha}^{*})_{\alpha,U}=(T^{*}({y_{\alpha}^{*}}_{|X}))_{\alpha,U}=
\widehat{T}({y_{\alpha}^{*}}_{|X})_{\alpha,U}\in\widehat{T}\widehat{X}$. Hence $\widehat{T}_{0}\widehat{Y}\subset\widehat{T}\widehat{X}$.

Conversely, if $(x_{\alpha}^{*})_{\alpha,U}\in\widehat{X}$ then we can extend each $x_{\alpha}^{*}$ to an element $y_{\alpha}^{*}\in Y^{*}$ 
such that $\|y_{\alpha}^{*}
\|=\|x_{\alpha}^{*}\|$. 
Thus we have $\displaystyle\left(y_{\alpha}^{*}-w^{*}\hyphen\lim_{\beta,U}y_{\beta}^{*}\right)_{\alpha,U}\in\widehat{Y}$. 
Note that 
\[T_{0}^{*}\left(w^{*}\hyphen\lim_{\beta,U}y_{\beta}^{*}\right)=w^{*}\hyphen\lim_{\beta,U}T_{0}^{*}y_{\beta}^{*}=w^{*}\hyphen\lim_{\beta,U}T^{*}x_{\beta}^{*}=T^{*}
\left(w^{*}\hyphen\lim_{\beta,U}x_{\beta}^{*}\right)=0.\]
This implies that
\begin{eqnarray*}
\widehat{T}(x_{\alpha}^{*})_{\alpha,U}=(T^{*}x_{\alpha}^{*})_{\alpha,U}&=&(T_{0}^{*}y_{\alpha}^{*})_{\alpha,U}\\&=&\left(T_{0}^{*}\left(y_{\alpha}^{*}-w^{*}\hyphen\lim_
{\beta,U}y_{\beta}^{*}\right)\right)_{\alpha,U}\\&=&\widehat{T}_{0}\left(y_{\alpha}^{*}-w^{*}\hyphen\lim_{\beta,U}y_{\beta}^{*}\right)_{\alpha,U}\in\widehat{T}_{0}
\widehat{Y}.
\end{eqnarray*}
Therefore $\widehat{T}\widehat{X}\subset\widehat{T}_{0}\widehat{Y}$.
\end{proof}
\begin{proposition}\label{pro2.5}
Suppose that $T\in B(X)$ and that there exists an essentially maximal subspace $M$ of $X$ on which $T$ is compact. Then $\widehat{T}$ has closed range.
\end{proposition}
\begin{proof}
Without loss of generality, we may assume that $X$ is a subspace of $Y=\ell_{\infty}(J)$ for some set $J$. Define $T_{0}\in B(X,\ell_{\infty}(J))$, $x\mapsto Tx$. Then by
assumption, there is an essentially maximal subspace $M$ of $X$ on which $T_{0}$ is compact. By \cite[Theorem 3.3]{Lindenstrauss}, there exists $K\in K(X,
\ell_{\infty}(J))$ such that ${K}_{|M}={T_{0}}_{|M}$.

We now show that $T_{0}-K\in B(X,\ell_{\infty}(J))$ has closed range. Since $M\subset\text{ker}(T_{0}-K)$ and $M$ is an essentially maximal subspace of $X$ on which
$T_{0}-K$ is compact, $\text{ker}(T_{0}-K)$ is an essentially maximal subspace of $X$ on which $T_{0}-K$ is compact.

Let $\pi$ be the quotient map from $X$ onto $X/\text{ker}(T_{0}-K)$. Define the (one-to-one) operator $R:X/\text{ker}(T_{0}-K)\to \ell_{\infty}(J)$, $\pi x\mapsto(T_{0}-K)
x$. If $R$ does not have closed range, then by \cite[Proposition 2.c.4]{LinTza}, $R$ is compact on an infinite dimensional subspace $V$ of $X/\text{ker}(T_{0}-K)$.
Hence, $T_{0}-K$ is compact on $\pi^{-1}V$ and so by the essential maximality of $\text{ker}(T_{0}-K)$, we have $\text{dim }\pi^{-1}V/\text{ker}(T_{0}-K)<\infty$. Thus,
$V=\pi^{-1}V/\text{ker}(T_{0}-K)$ is finite dimensional, which contradicts the definition of $V$.

Therefore, $R$ has closed range and so $T_{0}-K$ also has closed range. By Proposition \ref{pro2.3}, $\widehat{T_{0}-K}$ has closed range. But $\widehat{K}=0$ so
$\widehat{T}_{0}$ has closed range and by Lemma \ref{lem2.4}, $\widehat{T}$ has closed range.
\end{proof}
\begin{lemma}\label{lem2.6}
Suppose that $P\in B(X)$ and that $\dot{P}$ is an idempotent in $B(X)/K(X)$. Then $P$ is the sum of an idempotent in $B(X)$ and a compact operator on $X$.
\end{lemma}
\begin{proof}
We first treat the case where the scalar field is $\mathbb{C}$. From Fredholm theory (see e.g. \cite[Chapters XI and XVII]{Gohberg}), we know that since $\sigma(\dot{P})
\subset\{0,1\}$, the only possible cluster points of $\sigma(P)$ are 0 and 1. Thus, there exists $0<r<1$ such that $\{z\in\mathbb{C}:|z-1|=r\}\cap\sigma(P)=\emptyset$.
Then $\dot{P}=\frac{1}{2\pi i}\oint_{|z-1|=r}(z\dot{I}-\dot{P})^{-1}\,dz$ and so $P-\frac{1}{2\pi i}\oint_{|z-1|=r}(zI-P)^{-1}\,dz\in K(X)$. But $\frac{1}{2\pi i}\oint_{
|z-1|=r}(zI-P)^{-1}\,dz$ is an idempotent in $B(X)$ (see e.g. \cite[Theorem 2.7]{Radjavi}). This completes the proof in the complex case.

If $X$ is a real Banach space, then let $X_{C}$ and $P_{C}$ be the complexifications (see \cite[page 266]{Edelstein}) of $X$ and $P$, respectively. Thus, $\dot{P}_{C}$
is an idempotent in $B(X_{C})/K(X_{C})$. Since the only possible cluster points of $\sigma(P_{C})$ are 0 and 1, there exists a closed rectangle $R$ in the complex plane
symmetric with respect to the real axis such that 1 is in the interior of $R$, 0 is in the exterior of $R$, and $\sigma(P_{C})$ is disjoint from the boundary $\partial R$
of $R$. By \cite[Lemma 3.4]{Edelstein}, the idempotent $\frac{1}{2\pi i}\oint_{\partial R}(zI-P_{C})^{-1}\,dz$ in $B(X_{C})$ is induced by an idempotent $P_{0}$ in $B(X)$.
Since $P_{C}-\frac{1}{2\pi i}\oint_{\partial R}(zI-P_{C})^{-1}\,dz\in K(X_{C})$, we see that $P-P_{0}\in K(X)$.
\end{proof}
\begin{proof}[Proof of Theorem \ref{thm1.3}]
``(1)$\Rightarrow$(2)'':  Let $\displaystyle\dot{P}:=\lim_{n\to\infty}\frac{\dot{I}+\dot{T}+\ldots+\dot{T}^{n}}{n+1}$.\\Since
$\displaystyle\lim_{n\to\infty}\frac{\|\dot{T}^{n}\|}{n}=0$,
\begin{equation}\label{eq1}
(\dot{I}-\dot{T})\dot{P}=\lim_{n\to\infty}(\dot{I}-\dot{T})\frac{\dot{I}+\dot{T}+\ldots+\dot{T}^{n}}{n+1}=\lim_{n\to\infty}\frac{\dot{I}-\dot{T}^{n+1}}{n+1}=0.
\end{equation}
Thus $\dot{T}\dot{P}=\dot{P}$, and so
\[\dot{P}^{2}=\lim_{n\to\infty}\frac{\dot{P}+\dot{T}\dot{P}+\ldots+\dot{T}^{n}\dot{P}}{n+1}=\lim_{n\to\infty}\frac{(n+1)\dot{P}}{n+1}=\dot{P}.\]
Hence $\dot{P}$ is an idempotent in $B(X)/K(X)$. By Lemma \ref{lem2.6}, there exists an idempotent $P_{0}$ in $B(X)$ such that $P-P_{0}\in K(X)$. Replacing $P$ with
$P_{0}$, we can assume without loss of generality that $P$ is an idempotent in $B(X)$. Equation (\ref{eq1}) also implies that $(I-T)P\in K(X)$, which means that $I-T$ is
compact on $PX$. Hence $(I-T)^{m}$ is compact on $PX$.

We now show that $PX$ is an essentially maximal subspace of $X$ on which $(I-T)^{m}$ is compact. Suppose that $(I-T)^{m}$ is compact on a subspace $M_{0}$ of $X$
containing $PX$. Let\[f_{n}(z):=\frac{n+(n-1)z+(n-2)z^{2}+\ldots+z^{n-1}}{n+1},\quad z\in\mathbb{C},\,n\geq 1.\]
Note that $\dot{I}-\frac{\dot{I}+\dot{T}+\ldots+\dot{T}^{n}}{n+1}=(\dot{I}-\dot{T})f_{n}(\dot{T})$. Therefore,
\[\dot{I}-\dot{P}=(\dot{I}-\dot{P})^{m}=\lim_{n\to\infty}f_{n}(\dot{T})^{m}(\dot{I}-\dot{T})^{m},\]
and so
\[\lim_{n\to\infty}\|(I-P)-(f_{n}(T)^{m}(I-T)^{m}+K_{n})\|=0,\]
for some $K_{1},K_{2},\ldots\in K(X)$.

Since $(I-T)^{m}$ is compact on $M_{0}$, the operator $f_{n}(T)^{m}(I-T)^{m}$ is compact on $M_{0}$ and so is $f_{n}(T)^{m}(I-T)^{m}+K_{n}$ on $M_{0}$. Thus
${(I-P)}_{|M_{0}}$ is the norm limit of a sequence in $K(M_{0},X)$, and so $I-P$ is compact on $M_{0}$. Since $PX\subset M_{0}$, we have that $(I-P)M_{0}\subset M_{0}$.
Therefore, ${(I-P)}_{|(I-P)M_{0}}=I_{|(I-P)M_{0}}$ is compact, and so $(I-P)M_{0}$ is finite dimensional. In other words, $\text{dim }M_{0}/PX<\infty$.

``(2)$\Rightarrow$(1)'':  By Proposition \ref{pro2.5}, $\widehat{(I-T)^{m}}=(I-\widehat{T})^{m}$ has closed range. Since by assumption $\displaystyle\lim_{n\to\infty}
\frac{\|T^{n}\|_{e}}{n}=0$, $\displaystyle\lim_{n\to\infty}\frac{\|\widehat{T}^{n}\|}{n}=\lim_{n\to\infty}\frac{\|\widehat{T^{n}}\|}{n}=0$. By Mbekhta-Zem\'anek's Theorem
\ref{thm1.2}, the sequence $(M_{n}(\widehat{T}))_{n}$ converges in norm to an element in $B(\widehat{X})$. By Theorem \ref{thm2.1}, the result follows.
\end{proof}
\bibliographystyle{amsplain}

\end{document}